\documentclass[11pt,b5paper,twoside, headrule]{amsart}

\usepackage{amsfonts, amsmath, amssymb,latexsym}
\usepackage{epsfig}
\usepackage[curve]{xy}
\usepackage{algorithmic}
\usepackage{algorithm}
\usepackage{enumerate}
\usepackage{framed}
\usepackage{hyperref}

\footskip=10mm\parskip 0.2cm\addtocounter{page}{0}
\newtheorem{thm}{Theorem}[section]
\newtheorem{cor}[thm]{Corollary}
\newtheorem{lem}[thm]{Lemma}
\newtheorem{prop}[thm]{Proposition}
\newtheorem{conj}[thm]{Conjecture}

\theoremstyle{definition}
\newtheorem{defn}[thm]{Definition}

\theoremstyle{remark}

\numberwithin{equation}{section}

\title[On primes of the form $n^2+n+p$]{Primes of the form $n^2+n+p$ have density 1}

\author{\sc Iv\'an Blanco-Chac\'on}
\address{School of Mathematics and Statistics\\
University College Dublin\\
Belfield, Dublin 4\\
Ireland}
\email{ivan.blanco-chacon@ucd.ie}

\author{\sc Gary McGuire}
\address{School of Mathematics and Statistics\\
University College Dublin\\
Belfield, Dublin 4\\
Ireland}
\email{gary.mcguire@ucd.ie}

\author{\sc Oisin Robinson}
\address{School of Mathematics and Statistics\\
University College Dublin\\
Belfield, Dublin 4\\
Ireland}
\email{oisin.robinson@ucd.ie}

\begin{document}
\renewcommand\baselinestretch{1.2}
\renewcommand{\arraystretch}{1}
\def\base{\baselineskip}
\font\tenhtxt=eufm10 scaled \magstep0 \font\tenBbb=msbm10 scaled
\magstep0 \font\tenrm=cmr10 scaled \magstep0 \font\tenbf=cmb10
scaled \magstep0


\def\evenhead{{\protect\centerline{\textsl{\large{I. Blanco}}}\hfill}}

\pagestyle{myheadings} \markboth{\evenhead}{\oddhead}

\thispagestyle{empty}

\maketitle

\begin{abstract}
We consider the representation of primes in the form $n^2+n+p$, where $n$ is a positive integer and $p$ is a prime.
We prove that a subset of the primes having density 1 is representable in this form. 
We conjecture that every prime number $\geq 5$ is expressible in the form $n^2+n+p$ where $p$ is a twin prime, and we show that this conjecture implies the existence of infinitely many twin primes. 
\end{abstract}

\bigskip
\section{Introduction}

Our main result in this paper is the following theorem.

\begin{thm}\label{main}
There exists a density 1 subset $S$ of the prime numbers  such that for every $q\in S$, there exists a prime $p$ and an integer  $n\geq 1$ such that $q=n^2+n+p$.
\end{thm}

While Dirichlet's Theorem asserts that for each relatively prime pair $a,b\in\mathbb{N}$ the sequence $\{an+b\}_{n\geq 1}$ contains infinitely many prime numbers, the question of how often an irreducible quadratic polynomial with integer coefficients assumes a prime value when evaluated at natural numbers is still an open problem. In this regard,  Buniakowski's conjecture is well known.

\begin{conj}[Buniakowski] Let $f(x)\in\mathbb{Z}[x]$ be a quadratic irreducible polynomial such that the sequence $\{f(n)\}_{n\geq 1}$ does not have a common factor. Then, there exist infinitely many primes in the sequence.
\end{conj}

However,  it is even unknown if a general polynomial meeting the conditions as in Buniakowski's conjecture will always produce at least one prime number in the sequence $\{f(n)\}_{n\geq 1}$.

Notice that each polynomial $f_p(x)=x^2+x+p$, for primes $p\geq 3$, satisfies the conditions in Buniakowski's conjecture. 
Theorem \ref{main} says that a density 1 set of primes occur in the 
sequences $\{f_p(n)\}_{n\geq 1}$ as we range over all primes $p\geq 3$. Our proof also gives a lower bound for the number of primes in the sequence $\{f_p(n)\}_{n\geq 1}$ for each 
$p$. 
Such a lower bound was also given by Granville and Mollin  \cite{granville}; 
however our result uses a different region
and gives control of the error term which allows us to prove
Theorem \ref{main}.



Sun \cite{sun}  conjectures that any odd integer larger than 3 can be written in the form $n^2+n+p$, where $p$
is a prime. 
Theorem \ref{main} can be seen as evidence towards Sun's conjecture.

Our result can also be seen as evidence towards the following conjecture\footnote{This conjecture arose from an entry into the Ireland BT Young Scientist Competition} which originally motivated this work.

\begin{conj}\label{conj}
For every prime number $q\geq 5$, there exists a twin prime $p$ and an integer  $n\geq 1$ such that 
$q=n^2+n+p$.
\end{conj}

To sum up,
\begin{itemize}
\item Conjecture \ref{conj} states that the set of prime numbers is contained in the union of the sequences $\{f_p(n)\}_{n\geq 1}$ as $p$ runs over all twin primes. 
\item Theorem \ref{main} states that a density 1 subset of primes is contained in the union of the sequences $\{f_p(n)\}_{n\geq 1}$ as $p$ runs over all primes. 
\item Sun's conjecture states that all odd integers larger than 3 are 
in the union of the sequences $\{f_p(n)\}_{n\geq 1}$ as $p$ runs over all primes. 
\end{itemize}

This paper is laid out as follows.
In Section \ref{asy} we prove the following asymptotic formula, which apart from being interesting by itself, will allow us to prove Theorem \ref{main} in Section \ref{wrapup}.

\begin{thm} 
For a prime $p\geq 2$ let $\kappa=\kappa(p)=4p-1$, and fix  $A,B>0$.
For $x,y\in\mathbb{R}$ large enough 
in the region $x^2\log(x)^{-A}\leq y\leq x^2$ we have
$$
\sum_{\begin{array}{c}\kappa\leq y\\ \mu^2(\kappa)=1\end{array}}\left|\sum_{n\leq x}\Lambda(n^2+n+p)-S(\kappa)\frac{x}{2}\right|^2=O\left(\frac{yx^2}{(\log x)^B}\right).
$$
\label{r1intro}
\end{thm}

In Section \ref{twins} we prove that the validity of Conjecture \ref{conj} would imply the existence of infinitely many pairs of twin primes. What is more, the validity of our conjecture yields a lower bound for the asymptotic growth of the set of twin primes. 
Section \ref{compev} addresses computational evidence -- that
Conjecture \ref{conj} holds for the first 100 million primes.

Theorem \ref{r1intro} is an analogue to the main result in Baier and Zhao \cite{baier}, where the authors study the asymptotic distribution of primes expressible as $n^2+k$, for square free $k\geq 0$, in the same logarithmic region as we do. Our proof is a  modification of the proof in \cite{baier}; the control of the negligible terms relies on the same arguments as in their paper and so we refer the reader to \cite{baier} for that, while we derive in full detail the dominant term appearing in Theorem \ref{r1intro}, since our argument is not immediate from theirs.

Finally, we thank Liangyi Zhao and Nigel Boston and an anonymous aide for their kind and useful answers to our questions.

From now on, by a twin prime we mean a prime number $p$ such that either $p-2$ or $p+2$ is also prime. For $x\in\mathbb{R}$, let us denote by $\pi(x)$ the number of prime numbers less than or equal to $x$ and by $\pi_2(x)$, the number of twin primes less than or equal $x$. For functions $\psi,\phi:[0,\infty)\to\mathbb{R}$, we write $\psi(x)\gg\phi(x)$ if 
$\psi(x)\geq\phi(x)$ for $x$ big enough. In particular, it is well known that
$$
\pi(x)\gg\frac{x}{\log(x)}.
$$

\section{An asymptotic formula}\label{asy}

Throughout the rest of the paper, let $\Lambda$ be the von Mangoldt function, $\mu$ the M\"obius function, and for $k\geq 1$, denote
$$
S(k)=\prod_{p>2}\left(1-\frac{\left(\frac{-k}{p}\right)}{p-1}\right)
$$
where the product is over primes $p>2$ and $\left(\frac{-k}{p}\right)$ is the Legendre symbol.

The following result due to Baier and Zhao describes the asymptotic distribution of primes in the progressions $n^2+k$, for square free $k$:

\begin{thm}[Baier, Zhao \cite{baier}] Given $A,B>0$, for $x,y\in\mathbb{R}$ with $x^2\log(x)^{-A}\leq y\leq x^2$, and for $k\in\mathbb{N}$, it is
$$
\sum_{\begin{array}{c}k\leq y\\ \mu^2(k)=1\end{array}}\left|\sum_{n\leq x}\Lambda(n^2+k)-S(k)x\right|^2=O\left(\frac{yx^2}{(\log(x))^B}\right).
$$
\label{bz}
\end{thm}

Our aim in this section is to prove the following variation, from which we will deduce Theorem \ref{main} as a corollary:

\begin{thm} 
For a prime $p\geq 2$ let $\kappa=\kappa(p)=4p-1$.
Given $A,B>0$, for $x,y\in\mathbb{R}$ with $x^2\log(x)^{-A}\leq y\leq x^2$ we have
$$
\sum_{\begin{array}{c}\kappa\leq y\\ \mu^2(\kappa)=1\end{array}}\left|\sum_{n\leq x}\Lambda(n^2+n+p)-S(\kappa)\frac{x}{2}\right|^2=O\left(\frac{yx^2}{(\log(x))^B}\right).
$$
\label{r1}
\end{thm}

We give the proof of Theorem \ref{r1} in three subsections.

\subsection{Proof of Theorem \ref{r1} Part 1: The circle method}

We start writing

$$
\sum_{n\leq x}\Lambda(n^2+n+p)=\int_0^1\sum_{m\leq z}\Lambda(m)e(\alpha m)\sum_{n\leq x}e(-\alpha(n^2+n+p))d\alpha,
$$

where, throughout the paper, $e(u)=exp(2\pi i u)$, $z=x^2+x+y$. Next, we write:
$$
\sum_{n\leq x}\Lambda(n^2+n+p)=\int_0^1\sum_{m\leq z}\Lambda(m)e(\alpha m)\sum_{\begin{array}{c}n\leq 2x+1\\(n,2)=1\end{array}}e\left(\frac{-\alpha}{4}(n^2+\kappa)\right)d\alpha.
$$

For $x,y>0$ take $Q_1=\log(x)^c$ ($c$ fixed) and $Q=x^{1-\epsilon}$, so that if $x>>1$, $Q>Q_1$. Our argument, as Baier-Zhao's uses Hardy-Littlewood's circle method (\cite{circle}). We give the necessary notations next.

\begin{defn}The major arc is the set
$$
M=\bigcup_{q\leq Q_1}\bigcup_{\begin{array}{c}a=1\\(a,q)=1\end{array}}^{q}J_{q,a},
$$
where $J_{q,a}=\left[\frac{a}{q}-\frac{1}{qQ},\frac{a}{q}+\frac{1}{qQ}\right]$. The minor arc is the set
$$
\frak{m}=\left[\frac{1}{Q},1+\frac{1}{Q}\right]\setminus M.
$$
\end{defn}

The integral over the major arc will contribute the singular series, which gives rise to the main term, as well as a tail term $\Phi(\kappa)$ together with other error terms. The integral over the minor arc gives also error terms. All error terms are controlled in Baier-Zhao's paper and we will refer the reader to the relevant sections of the work. Next we derive the dominant term in Theorem \ref{r1intro} from the integral over the major arc.

Let us denote

$$
S_1(\alpha)=\sum_{m\leq z}\Lambda(m)e(\alpha m),\;\;\;S_2(\alpha)=\sum_{\begin{array}{c}r\leq 2x+1\\\mbox{ odd}\end{array}}e\left(-\frac{\alpha}{4} r^2\right).
$$

For $\alpha=\frac{a}{q}+\beta\in M$, it is

$$S_1(\alpha)=\sum_{\begin{array}{c}m\leq z\\(m,q)=1\end{array}}\Lambda(m)e\left(\frac{am}{q}\right)e(\beta m)+O(log(z)^2).$$

Likewise, for $(am,q)=1$, 

\begin{equation}
e\left(\frac{am}{q}\right)=\frac{1}{\phi(q)}\sum_{\chi\pmod{q}}\chi(am)\tau(\overline{\chi}),
\label{eq1}
\end{equation}

where for a Dirichlet character $\chi$, $\tau(\chi)$ stands for its attached Gauss sum and $\phi$ is the Euler totient function. By using (\ref{eq1}), we can write

$$
S_1(\alpha)=T_1(\alpha)+E_1(\alpha)+O(log(z)^2),
$$
where
$$
T_1(\alpha)=\frac{\mu(q)}{\phi(q)}\sum_{m\leq z}e(\beta m),
$$
and
$$
E_1(\alpha)=\frac{1}{\phi(q)}\sum_{\chi\pmod{q}}\tau(\overline{\chi})\chi(a)\sum_{m\leq z}'\chi(m)\Lambda(m)e(\beta m),
$$
the $'$ meaning that for the trivial character, what appears in the sum is $\Lambda(m)-1$.

Next, also by using (\ref{eq1}), we can write
\begin{equation}
\sum_{d\mid 4q}\frac{1}{\phi(q_1^*)}\sum_{\chi\pmod{q_1^*}}\tau(\overline{\chi})\chi(-ad^*)\sum_{\begin{array}{c}r\leq 2x+1\\(r,4q)=d\\r\mbox{ odd}\end{array}}\chi^2(r^*)e\left(-\frac{\beta}{4} r^2\right),
\label{eq1coma5}
\end{equation}

with $r^*=r/d$, $q^*=4q/d$, $d^*=d/(d,q^*)$, $q_1^*=q^*/(d,q^*)$. Now, observe crucially (this does not happen in Baier-Zhao's proof), that in the most inner sum in (\ref{eq1coma5}), the indices over which we are summing are odd. Since the gcd's of these $r$ with $4q$ are $d$, it turns out that the $d's$ we are summing over in the outer sum are also odd, hence, the above expression becomes:

$$
\sum_{d\mid q}\frac{1}{\phi(q_1^*)}\sum_{\chi\pmod{q_1^*}}\tau(\overline{\chi})\chi(-ad^*)\sum_{\begin{array}{c}r\leq 2x+1\\(r,4q)=d\\r\mbox{ odd}\end{array}}\chi^2(r^*)e\left(-\frac{\beta}{4} r^2\right),
$$

And this, can also be decomposed as

$$S_2(\alpha)=T_2(\alpha)+E_2(\alpha),$$

where

$$
T_2(\alpha)=\sum_{d\mid q}\frac{1}{\phi(q_1^*)}\sum_{\begin{array}{c}l=1\\(l,q_1^*)=1\end{array}}^{q_1^*}e\left(\frac{-ad^*l^2}{q_1^*}\right)\sum_{\begin{array}{c}r\leq 2x+1\\\mbox{ odd}\\(r,4q)=d\end{array}}e\left(-\frac{\beta}{4} r^2\right),
$$

and

$$
E_2(\alpha)=\sum_{d\mid q}\frac{1}{\phi(q_1^*)}\sum_{\begin{array}{c}\chi\pmod{q_1^*}\\\chi^2\neq 1\end{array}}\tau(\overline{\chi})\chi(-ad^*)\sum_{\begin{array}{c}r\leq 2x+1\\(r,4q)=d\\r\mbox{ odd}\end{array}}\chi^2(r^*)e\left(-\frac{\beta}{4} r^2\right).
$$

All told, the integral $\displaystyle\int_M\sum_{m\leq z}\Lambda(m)e(\alpha m)\sum_{n\leq x}e(-\alpha(n^2+n+p))d\alpha$ has been decomposed as

$$
\int_M(T_1(\alpha)+E_1(\alpha)+O(log(x)^2))(T_2(\alpha)+E_2(\alpha))e\left(-\frac{\alpha}{4}\kappa\right)d\alpha.
$$

\subsection{Proof of Theorem \ref{r1} Part 2: The singular series}

\bigskip

Next, we will prove that 
$$
\int_MT_1(\alpha)T_2(\alpha)e\left(-\frac{\alpha}{4}\kappa\right)d\alpha=S(\kappa)\frac{x}{2}+O\left(\sqrt{x}\log(x)^{c_1}\right),
$$
as the rest of integrals are negligible, as proved in \cite{baier}.

So, we start by expressing $\displaystyle\int_MT_1(\alpha)T_2(\alpha)e\left(-\frac{\alpha}{4}\kappa\right)d\alpha$ as

\begin{equation}
\sum_{q\leq Q_1}\frac{\mu(q)}{\phi(q)}\sum_{\begin{array}{c}a=1\\(a,q)=1\end{array}}^{q}e\left(-\frac{a\kappa}{4q}\right)\sum_{d\mid q}\frac{G(a,q_1^*)}{\phi(q^*_1)}\int_{|\beta|\leq 1/qQ}\Pi_{q,d}(\beta)d\beta,
\label{integral}
\end{equation}

with

\begin{equation}
G(a,q_1^*)=\sum_{\begin{array}{c}l=1\\(l,q_1^*)=1\end{array}}^{q_1^*}e\left(\frac{-ad^*l^2}{q_1^*}\right)
\label{pre1}
\end{equation}

and

\begin{equation}
\Pi_{q,d}(\beta)=\sum_{m\leq z}e(\beta m)\sum_{\begin{array}{c}r\leq 2x+1\\(r,4q)=d\\r\mbox{ odd}\end{array}}e\left(-\frac{\beta}{4} r^2\right)e\left(-\frac{\kappa\beta}{4}\right).
\label{pre2}
\end{equation}

Now, using the bound for the geometric series over $m$ in (\ref{pre2}), Cauchy and Parseval inequalities, the integral of the right hand side of (\ref{integral}) is expressed as

\begin{equation}
\int_0^1\sum_{m\leq z}e(\beta m)\sum_{\begin{array}{c}r\leq 2x+1\\r\mbox{ odd}\\(r,4q)=d\end{array}}e\left(-\frac{\beta r^2}{4}\right)e\left(-\frac{\kappa\beta}{4}\right)d\beta+O\left(\left(qQ\frac{x}{d}\right)^{1/2}\right).
\label{mainterm}
\end{equation}

The first term of (\ref{mainterm}) is, by orthogonality of the exponential

$$
\sum_{\begin{array}{c}m\leq z\\m=r^2+\kappa\end{array}}\sum_{\begin{array}{c}r\leq 2x+1\\(r,4q)=d\\r\mbox{ odd}\end{array}}1=\sum_{\begin{array}{c}r^*\leq (2x+1)/d\\(r^*,4q/d)=1\\r^*\mbox{  odd}\end{array}}1=\frac{\phi(4q/d)}{8q}(2x+1)+O\left(\phi(4q/d)\right).
$$

Since $\phi(4q/d)\ll\sqrt{qQ}$, the integral \ref{integral} becomes

\begin{equation}
\sum_{q\leq Q_1}\frac{\mu(q)}{\phi(q)}\sum_{\begin{array}{c}a=1\\(a,q)=1\end{array}}^{q}e\left(-\frac{a\kappa}{4q}\right)\sum_{d\mid q}\frac{G(a,q_1^*)}{\phi(q^*_1)}\left(\frac{\phi(4q/d)}{8q}(2x+1)+O\left((qQx/d)^{1/2}\right)\right).
\label{integral2}
\end{equation}

We only have to consider square free $q's$, since $\mu(q)=0$ otherwise. For these $q$'s, we have $d^*=d$, $q_1^*=4q/d$, and (\ref{integral2}) becomes

\begin{equation}
\frac{x}{4}\sum_{q\leq Q_1}\frac{\mu(q)}{q\phi(q)}\sum_{\begin{array}{c}a=1\\(a,q)=1\end{array}}^{q}e\left(-\frac{a\kappa}{4q}\right)\sum_{d\mid q}\sum_{\begin{array}{c}l=1\\(l,4q/d)=1\end{array}}^{4q/d}e\left(\frac{-a(dl)^2}{4q}\right)+O(\sqrt{xQ}log(x)^{c_1})
\label{integral3}
\end{equation}

for some fixed $c_1>0$.

Next, we study the term $\Sigma(q):=\displaystyle \sum_{\begin{array}{c}a=1\\(a,q)=1\end{array}}^{q}e\left(-\frac{a\kappa}{4q}\right)\sum_{d\mid q}\sum_{\begin{array}{c}l=1\\(l,q^*)=1\end{array}}^{4q/d}e\left(\frac{-a(dl)^2}{4q}\right)$ in (\ref{integral3}), which can easily be expressed as

$$
\sum_{\begin{array}{c}r=1\\(r,2)=1\end{array}}^{4q}\sum_{\begin{array}{c}a=1\\(a,q)=1\end{array}}^{q}e\left(-\frac{a}{4q}(\kappa+r^2)\right).
$$
The treatment of these terms is also one of the key differences between our work and Baier-Zhao's.

First, notice that the only $r's$ that remain in the outer sum above are the odd ones, since in (\ref{integral3}), the $l's$ in the inner sum are the odd ones, and $d$ is also odd. In particular, this means that the numbers $\kappa+r^2$ are divisible by $4$.

\begin{prop}If $q$ is prime, then
$$
\Sigma(q)=\left\lbrace\begin{array}{l}2q\left(\frac{-\kappa}{q}\right),\mbox{ if }q\neq2,\\0\mbox{ if }q=2.\end{array}\right.
$$
with $\left(\frac{a}{b}\right)$ the Legendre symbol.
\label{legendreP}
\end{prop}

\begin{proof}
Suppose that $q$ is an odd prime. We have

$$
\sum_{\begin{array}{c}a=1\\(a,q)=1\end{array}}^{q}e\left(-\frac{a}{4q}(\kappa+r^2)\right)=\sum_{a=1}^{q}e\left(-\frac{a}{4q}(\kappa+r^2)\right)-1,
$$

Hence,

$$
\Sigma(q)=\sum_{\begin{array}{c}r=1\\4q\mid\kappa+r^2\end{array}}^{4q}\sum_{a=1}^{q}e\left(-\frac{a}{4q}(\kappa+r^2)\right)-2q,
$$

i.e

$$
\Sigma(q)=\sum_{\begin{array}{c}r=1\\4q\mid\kappa+r^2\end{array}}^{4q}q-2q.
$$

Now, notice that given $r\in\{1,...,q\}$ such that $q\mid\kappa+r^2$, we have two odd $r_1,r_2\in\{1,...,4q\}$ such that $q\mid\kappa+r_i^2$ ($i=1,2$) and the $r_i$ are congruent modulo $q$: if $r$ is odd, take $r_1=r$ and $r_2=2q+r$, otherwise, take $r_1=q+r$ and $r_2=3q+r$. Notice that $r$ odd is equivalent to $4\mid\kappa+r^2$. Reciprocally, odd $r's$ in $\{1,...,4q\}$ with $q\mid\kappa+r^2$ come in pair (both elements in each pair are congruent mod $q$). Hence,

$$
\Sigma(q)=\sum_{\begin{array}{c}r=1\\q\mid\kappa+r^2\end{array}}^{q}2q-2q=2q\left(\frac{-\kappa}{q}\right).
$$

For $q=2$, the left sum runs through the odd indices $r\in\{1,...,8\}$ such that $8\mid r^2+\kappa$, which is the empty set, hence this sum is $0$ and so is $\Sigma(2)$.

\end{proof}

Now, we can prove the following characterisation of $\Sigma(q)$, for general square free odd $q\geq 1$.

\begin{thm}For a square free $q\geq 3$, we have
$$
\Sigma(q)=\left\lbrace\begin{array}{l}2q\left(\frac{-\kappa}{q}\right),\mbox{ if }2\nmid q,\\0\mbox{ if }2\mid q.\end{array}\right.
$$
with $\left(\frac{a}{b}\right)$ the Jacobi symbol.
\end{thm}

\begin{proof}
We need to check that the assignment $n\mapsto \Sigma(n)/2$ is a multiplicative arithmetic function. For this, it is convenient to observe first that, by setting $r=2n+1$, we have
$$
\Sigma(q)=2\sum_{n=0}^{q-1}\sum_{\begin{array}{c}a=1\\(a,q)=1\end{array}}^{q}e\left(-\frac{a}{q}(p+n^2+n)\right).
$$
So, for odd integers $q_1,q_2>3$ with $(q_1,q_2)=1$, we have
$$
\Sigma(q_1)/2=\sum_{n_1=0}^{q_1-1}\sum_{\begin{array}{c}a_1=1\\(a_1,q_1)=1\end{array}}^{q_1}e\left(-\frac{a_1}{q_1}(p+(q_2n_1)^2+(q_2n_1))\right),
$$
and analogously for $\Sigma(q_2)$, since being $(q_1,q_2)=1$, if $n_1$ runs over all residue classes modulo $q_1$, $q_2n_1$ do so, and analogously for $q_1n_2$. Hence

\begin{equation}
\Sigma(q_1)\Sigma(q_2)/4=\sum_{n_1=0}^{q_1-1}\sum_{\begin{array}{c}a_1=1\\(a_1,q_1)=1\end{array}}^{q_1}\sum_{n_2=0}^{q_2-1}\sum_{\begin{array}{c}a_2=1\\(a_2,q_2)=1\end{array}}^{q_2}e(f(a_1,a_2,q_1,q_2,r_1,r_2)),
\label{multiplicative}
\end{equation}

with

$$
f(a_1,a_2,q_1,q_2,r_1,r_2)=-p\frac{a_1q_2+a_2q_1}{q_1q_2}-\frac{a_1q_2(q_2n_1)^2+a_2q_1(q_1n_2)^2}{q_1q_2}-\frac{a_1q_2^2n_1+a_2q_1^2n_2}{q_1q_2}
$$

which is congruent, modulo $1$, with

$$
-p\frac{a_1q_2+a_2q_1}{q_1q_2}-\frac{(a_1q_2+a_2q_1)((n_1q_2+n_2q_1)^2+(n_1q_2+n_2q_1))}{q_2q_1}.
$$

Hence, if both $q_1$ and $q_2$ are odd, Equation \ref{multiplicative} yields

$$
2\Sigma(q_1q_2)=\Sigma(q_1)\Sigma(q_2).
$$

It is also easy to check with a similar argument as in Prop. \ref{legendreP}, that for $q$ odd prime, $\Sigma(2q)=0$.
\end{proof}

Hence, we have shown that (\ref{integral3}) becomes

$$
\frac{x}{2}\sum_{\begin{array}{c}q\leq Q_1\\2\nmid q\end{array}}\frac{\mu(q)}{\phi(q)}\left(\frac{1-4p}{q}\right)+O(\sqrt{xQ}log(x)^{c_1}),
$$

which, setting $\kappa=\kappa(p)=1-4p$, can be rewritten as

$$
S(\kappa)\frac{x}{2}-\Phi(Q_1,p)\frac{x}{2}+O(\sqrt{xQ}log(x)^{c_1}),
$$

with

$$
S(\kappa)=\sum_{\begin{array}{c}q=1\\2\nmid q\end{array}}^{\infty}\frac{\mu(q)}{\phi(q)}\left(\frac{1-4p}{q}\right)=\prod_{\begin{array}{c}q>2\\q\mbox{ prime}\end{array}}1-\frac{\left(\frac{1-4p}{q}\right)}{q-1}
$$

and

$$
\Phi(Q_1,p)=\sum_{\begin{array}{c}q>Q_1\\2\nmid q\end{array}}^{\infty}\frac{\mu(q)}{\phi(q)}\left(\frac{1-4p}{q}\right).
$$

Putting the results of this section all together, we arrive to

\begin{equation}
\int_MT_1(\alpha)T_2(\alpha)e(-\kappa\alpha)d\alpha=S(p)\frac{x}{2}+O\left(x|\Phi(Q_1,p)|\sqrt{xQ}log(x)^{c_1}\right).
\label{majorarc}
\end{equation}

\subsection{Proof of Theorem \ref{r1} Part 3:  The error terms}

First, we have:

\begin{prop}[\cite{baier}, Section 5] The following inequality holds:
$$
\sum_{\begin{array}{c}k\leq y\\ \mu^2(k)=1\end{array}}|\Phi(Q_1,p)|^2\ll\frac{y}{\log(x)^c}.
$$
\end{prop}

Second, analogously to \cite{baier}, by a similar use of Polya-Vinogradov's estimate, Gallagher's Lemma and Cauchy, Parseval and Besser inequalities, we obtain:

\begin{prop}The following estimate holds:
$$
\sum_{\kappa\leq y}\left|\int_M\left(T_1(\alpha)E_2(\alpha)+T_2(\alpha)E_1(\alpha)+E_1(\alpha)E_2(\alpha)\right)e(-\kappa\alpha)d\alpha\right|^2 \ll
$$
$$
\frac{x^5(\log x)^{c_4}}{Q^2}+\frac{x^4}{(\log(x))^{c_5}}+\frac{yx^3}{Q^2},
$$
for fixed $c_4$ and $c_5$.
\end{prop}

Last, by using Bessel's lemma, Parseval's inequality and Dirichlet approximation in exactly the same way as in \cite{baier}, we can control the integral over the minor arc:

\begin{prop}The following estimate holds:
$$
\sum_{\kappa(p)\leq y}\left|\int_{\frak{m}}\sum_{m\leq z}\Lambda(m)e(\alpha m)\sum_{n\leq x}e(-\alpha(n^2+n+p))\right|^2\leq \frac{x^4}{\log(x)^{c-3}}+\log(x)^3Qx^3.
$$
\end{prop}

This concludes the proof of Theorem \ref{r1}.

\section{Proof of Theorem \ref{main}}\label{wrapup}

We end our work by proving Theorem \ref{main} as a corollary of Theorem \ref{r1}. 
Firstly, let us make precise what we mean by density in this work.

\begin{defn}[Powell, \cite{powell}] A subset of prime numbers $S$ is said to have density\footnote{Notice that what we define here as \emph{density} is defined as \emph{primitive density}  in \cite{powell}, in other words, by density of a subset $S$ of prime numbers we mean the relative density of $S$ with respect to the set of primes.} $\delta$ if
$$
\lim_{x\to\infty}\frac{|\{p\leq x, p\in S\}|}{\pi(x)}=\delta.
$$
\end{defn}

Secondly, we need the following from Mirsky \cite{myrski}.

\begin{lem}Denote by $s(y)$ the number of primes $p\leq y$ such that $4p-1$ is square-free. 
Then we have
$$
s(y)\gg y/\log(y).
$$
\label{lemmamyrski}
\end{lem}
\begin{proof}This is Theorem 2 of \cite{myrski} for $k=2$.
\end{proof}

\begin{cor}
There exists a density $1$ subset $S$ of the set of prime numbers such that for every $q\in S$, there exists a prime $p$ and an integer  $n\geq 1$ such that $q=n^2+n+p$.
\end{cor}

\begin{proof}
We start with Theorem \ref{r1}.
Let $x,y\in\mathbb{R}$ with $x^2\log(x)^{-A}\leq y\leq x^2$.
The equality in the statement of  Theorem \ref{r1} also holds if in the inner sum we replace $\Lambda(n^2+n+p)$ by $\log(n^2+n+p)$ and keep only the $n$'s such that $n^2+n+p$ is prime. 
For, the remaining prime powers (on which the $\Lambda$ function is supported) will contribute $O\left(\frac{\sqrt{x}}{\log(x)}\right)$ in the inner sum before squaring, and hence $O\left(\frac{x^{3/2}}{\log(x)^2}\right)$ in the inner sum after squaring, hence $O\left(\frac{x^{3/2}y}{\log(x)^2}\right)$ over the full sum, which is negligible
in comparison to $O\left(\frac{x^2y}{(\log(x))^B}\right)$.
Thus

\begin{equation}\label{key}
\sum_{\begin{array}{c}\kappa\leq y\\ \mu^2(\kappa)=1\end{array}}
\left|\sum_{\begin{array}{c}n\leq x\\ n^2+n+p \textrm{ prime}\end{array}}  \log (n^2+n+p)-
S(\kappa)\frac{x}{2}\right|^2=O\left(\frac{x^2y}{(\log(x))^B}\right).
\end{equation}

Take $\kappa=4p-1$, with $p$ prime. If for $x\gg 1$ there does not exist $n$ with $2n+1\leq x$ and such that $n^2+n+p$ is prime, then the inner sum corresponding to such $p$ is 
$|S(\kappa)|x/2$. 
As in the introduction of \cite{baier},
$|S(\kappa)|\gg 1/\log(x)$. 

Thus, letting $N(y)$ be the number of primes $p\leq y/4$ such that $4p-1$ is square-free and $n^2+n+p$ is not prime for any $n$, by  \eqref{key} we have
$$
N(y)x^2/\log(x)^2\ll x^2y/(\log(x))^B.
$$
giving
$$
N(y)\ll y/(\log(x))^{B-2}.
$$
Because of the domain to which $x$ and $y$ belong we get 
$$
N(y) \ll y/(\log(y))^{B-2}.
$$

By Lemma \ref{lemmamyrski}  the number of primes $p\leq y/4$ such that $4p-1$ is square-free and $n^2+n+p$ is prime for some $n$ is   
$\gg y/\log(y)-N(y).$
Taking $B=4$ we get that the number of these primes is
$$\gg y/log(y)-y/\log(y)^2.$$ 
Dividing by $\pi(y)$, we obtain, by using the prime number theorem, that the density of these primes is $1$.
\end{proof}

\section{Relation of  Conjecture \ref{conj} with some open questions}\label{twins}

In the previous section we showed that 'most' primes can be written
as $n^2+n+p$ where $p$ is a prime.
Recall Conjecture \ref{conj}, which is the stronger statement that every prime $q$ can be written as $n^2+n+p$ where 
$p$ is a twin prime.

\subsection{A conjectural lower bound for twin primes}

For a prime $q\geq 5$ such that $q=p+n(n+1)$ with $p$ a twin prime and $n\geq 1$, let us denote by $p_q$ the smallest twin prime such that there exists some $n_q\geq 1$ with $q=p_q+n_q^2+n_q$. Notice that $p_q$ determines $n_q$: indeed, if $p_q+n_q^2+n_q=p_q+m_q^2+m_q$, since the function $x^2+x$ is injective for $x\geq 0$, we have $n_q=m_q$. 

Next, we give an easy lemma.

\begin{lem}Let $q'<q$ be prime numbers expressible as $q=p_q+n_q(n_q+1)$ and $q'=p_{q'}+n_{q'}(n_{q'}+1)$, with minimal $p_q$ and $p_{q'}$. If $n_q=n_{q'}$, then $p_{q'}<p_q$.
\label{easylemma}
\end{lem}
\begin{proof}If $n_q=n_{q'}$, then
$$
q=p_q+n_q^2+n_q=p_q+n_{q'}^2+n_{q'}=p_q-p_{q'}+q',
$$
hence $0<q-q'=p_q-p_{q'}$.
\end{proof}

\begin{thm}If Conjecture \ref{conj} is true, then 
$$
\pi_2(x)\gg\frac{\sqrt{x}}{\log(x)}.
$$
In particular, there exist infinitely many twin primes.
\label{implica}
\end{thm}
\begin{proof}Assume the conjecture to be true and for every prime number $q\geq 5$, consider the 
mapping $q\mapsto p_q$. First, the number of integers $m$ such that $m^2+m\leq x$ is $O(\sqrt{x})$. Second, let us split the set of primes less than or equal $x$ as the union of the sets $C_p(x):=\{q\leq x: p_q=p\}$, which are disjoint because of Lemma \ref{easylemma}. Then, for each twin prime $p\leq x$, $|C_p|$ is $O(\sqrt{x})$ and
$$
\pi(x)\ll\pi_2(x)\sqrt{x}.
$$
Since $\pi(x)\gg\frac{x}{\log(x)}$, the result follows.
\end{proof}




\section{Computational evidence}\label{compev}

Our computer verifications using Pari/GP show that, for the first 100 million primes $q$, there exists a twin prime $p_q$ and an integer $n_q\geq 1$ such that $q=p_q+n_q^2+n_q$ with $p_q>q^{1/3}$ and such that $n_q$ does not grow very fast (indeed it seems to grow with the logarithm of $q$). The fact of $p_q$ growing with $q$ is in accordance with Theorem \ref{implica}. In fact, for primes $q$ satisfying our conjecture, we have the following bounds.

\begin{lem}If Conjecture \ref{conj} is true, we have that, for primes $q$ large enough,
$$
n_q\leq\sqrt{q}.
$$
\label{easylemma2}
\end{lem}
\begin{proof}
If Conjecture \ref{conj} holds, the for a prime $q\geq 5$, we have $q=p_q+n_q(n_q+1)\geq 3+n_q^2+n_q$, hence $n_q\leq \frac{-1+\sqrt{4q-11}}{2}\sim\sqrt{q}$, for $q$ big enough.
\end{proof}

As mentioned before, our computer tests outperform this upper-bound. As for a conditional lower bound for $n_q$ we have the following:

\begin{prop}For $q\geq 5$ satisfying Conjecture \ref{conj}, either $p_q\geq q/2$ or $n_q\in[\frac{\sqrt{2q}}{2},\sqrt{q}]$.
\label{cases2}
\end{prop}

\begin{proof}For a prime $q$ such that $q=p_q+n_q^2+n_q$, by Lemma \ref{easylemma2}, $n\leq\sqrt{q}$, for $q$ big enough. On the other hand, either $p_q\leq q/2$ or $n_q^2+n_q\leq q/2$. When $n_q^2+n_q\leq q/2$, $q\leq p_q+q/2$, hence $p_q\geq q/2$, and as it is easy to see, these two facts are equivalent. If, on the other hand, $p_q\leq q/2$, it follows that $n_q^2+n_q=q-p_q\geq q/2$, hence $n_q\geq \frac{-1+\sqrt{1+2q}}{2}\sim\sqrt{2}\sqrt{q}/2$, for $q$ large.
\end{proof}

\end{document}